\input ./style/arxiv-vmsta.cfg
\documentclass[numbers,compress,v1.0.1]{vmsta}
\usepackage{vtexbibtags}
\usepackage{authorquery}
\usepackage[left]{vmkcol}

\volume{5}
\issue{3}
\pubyear{2018}
\firstpage{385}
\lastpage{410}
\aid{VMSTA113}
\doi{10.15559/18-VMSTA113}
\articletype{research-article}

\startlocaldefs
\ifdefined\HCode 
\def\texbold#1{#1}
\else 
\def\texbold#1{\textbf{#1}}
\fi

\newcommand{\rrvert}{\vert}
\newcommand{\llvert}{\vert}
\allowdisplaybreaks

\newtheorem{thm}{Theorem}
\newtheorem{tthm}{Theorem}

\newtheorem{lemma}{Lemma}
\newtheorem{cor}{Corollary}
\newtheorem{ccor}{Corollary}

\newtheorem*{CCOR}{Corollary A$'$}

\theoremstyle{definition}

\newtheorem{remark}{Remark}

\endlocaldefs

\begin{aqf}
\querytext{Q1}{Isn't here too many words on the uniform metric?}
\querytext{Q2}{Is it clear for a reader how they are defined?}
\end{aqf}
\begin{document}

\begin{frontmatter}
\pretitle{Research Article}

\title{On a bound of the absolute constant in the Berry--Esseen
inequality for i.i.d. Bernoulli random variables}

\author[a]{\inits{A.}\fnms{Anatolii}~\snm{Zolotukhin}\ead[label=e1]{zolot\_aj@mail.ru}} 
\author[b]{\inits{S.}\fnms{Sergei}~\snm{Nagaev}\ead[label=e2]{nagaev@math.nsc.ru}}
\author[c]{\inits{V.}\fnms{Vladimir}~\snm{Chebotarev}\thanksref{cor1}\ead[label=e3]{chebotarev@as.khb.ru}}
\thankstext[type=corresp,id=cor1]{Corresponding author.}
\address[a]{\institution{Tula State University}, \cny{Russian Federation}}
\address[b]{\institution{Sobolev Institute of Mathematics}, \cny{Russian Federation}}
\address[c]{\institution{Computing Center FEB RAS}, \cny{Russian Federation}}



\markboth{A. Zolotukhin et al.}{On a bound of the
absolute constant in the Berry\,--\,Esseen inequality $\ldots$}

\begin{abstract}
It is shown that the absolute constant  in the Berry--Esseen
inequality for i.i.d. Bernoulli random variables is  strictly less
than the Esseen constant, if $1\le n\le 500000$, where $n$ is a
number of summands. This result is got both with the help of a
supercomputer and  an interpolation theorem, which is proved in the
paper as well. In addition, applying the method developed by
S.~Nagaev and V.~Chebotarev in 2009--2011, an upper bound is obtained
for the absolute constant in the Berry--Esseen inequality in the
case under consideration, which differs from the Esseen constant by
no more than 0.06\%. As an auxiliary result, we prove a bound in the
local Moivre--Laplace theorem which has a simple and explicit
form.

Despite the best possible result,  obtained by J.~Schulz in 2016, we
propose our approach to the problem of finding the absolute constant
in the Berry--Esseen inequality for two-point distributions
since this approach, combining analytical methods and the use of
computers,  could be useful in solving other mathematical problems.
\end{abstract}

\received{\sday{30} \smonth{1} \syear{2018}}
\revised{\sday{22} \smonth{8} \syear{2018}}
\accepted{\sday{25} \smonth{8} \syear{2018}}
\publishedonline{\sday{14} \smonth{9} \syear{2018}}
\end{frontmatter}

\section{Introduction}
Let us consider the class $V$ of all probability distributions on
the real line $\mathbb R$, which have zero mean, unit variance and
finite third absolute moment.
 Let
$X,\,X_1,\,X_2,\,\ldots\,,X_n$ be i.i.d. random variables, where the
distribution of $X$ belongs to $V$. Denote
\begin{align*}
\varPhi(x)=\frac{1}{\sqrt{2\pi}}\int\limits
_{-\infty}^x e^{-t^2/2}\,dt,
\qquad \beta_3={\bf E}|X|^3.
\end{align*}
According to the
Berry--Esseen inequality \cite{Berry,Ess42}, there exists such
an absolute constant $C_0$ that  for all $n=1,\,2,\,\ldots\;$,
%
\begin{equation}
\label{B-E-ineq}\sup_{x\in{\mathbb R}} \Bigg|{\bf P} \Biggl(\frac{1}{\sqrt{n}}\sum
_{j=1}^nX_j<x \Biggr)-
\varPhi(x) \Bigg|\le \frac{C_0\beta_3}{\sqrt{n}}.
\end{equation}

The first upper bounds for the constant $ C_0 $ were obtained by
C.-G.~Esseen~\cite{Ess42} (1942), H.~Bergstr\"om~\cite{Berg49} (1949)
and K.~Takano~\cite{Takano} (1951).

In 1956 C.-G.~Esseen  \cite{Ess} showed that
%
\begin{equation}
\label{C0>}\lim_{n\to\infty}\frac{\sqrt{n}}{\beta_3}\sup
_{x\in\mathbb
R} \Bigg|{\bf P} \Biggl(\frac{1}{\sqrt{n}}\sum
_{j=1}^nX_j<x \Biggr)-\varPhi(x) \Bigg|\le
C_E,
\end{equation} where
$C_E=\frac{3+\sqrt{10}}{6\sqrt{2\pi}}=0.409732\,\ldots\;$.
 He has also found a two-point distribution, for
which the equality holds in \eqref{C0>}. He has proved the
uniqueness of such a distribution (up to a reflection).

 Consequently, $C_0\ge C_E$. The result of Esseen
served as an argument for the conjecture
%
\begin{equation}
\label{C=C}C_0= C_E,
\end{equation} that V.M. Zolotarev advanced in 1966
 \cite{Zolot66-TV}. The question whether the conjecture is correct  remains open up to now.

Since then, a number of upper bounds for $ C_0 $ have been obtained. A
historical review  can be found, for example, in
\cite{KorShv,preprint-2009,Shv}. We only note that recent results in this
field were obtained by I.S. Tyurin (see, for example,
\cite{arxive-Tyurin-2009,Tyurin-2009-Dokl,Tyurin-2010-TV,Tyurin-2010-Uspehi,Tyurin-Prob.Let.-2012}),
V.Yu. Korolev and I.G. Shevtsova (see, for example,
\cite{KorShv,KorSh-2010-TV}\xch{),}{,} and    I.G.~Shevtsova (see, for example,
\cite{arxive-Shevtsova-2013,Shevtsova-2011-TV,Shevtsova-2013-Inform,Shv,Shevtsova-2006-TV}).
The best upper estimate, known to date, belongs to Shevtsova: $C_0 \le0.469 $
\cite{Shv}. Note that in obtaining upper bounds, beginning from the estimates
in \cite{Zolot66-TV,Zolot67-ZW}, calculations play an essential role. In
addition, because of the large amount of computations, it was necessary to
use  computers.

The present paper is devoted to estimation of $C_0$  in the
particular case of i.i.d. Bernoulli random variables. In this case
we will use the notation $C_{02}$ instead of $C_0$. Let us recall
the chronology of the results along these lines.

In  2007 C.~Hipp and L.~Mattner  published  an analytical  proof of
the inequality $C_{02}\le\frac{1}{\sqrt{2\pi}}$ in the symmetric case
\cite{XiM.L.}.

In 2009  the second and third authors of the present paper
 have suggested the compound method
in which a refinement of C.L.T. for i.i.d. Bernoulli random
variables was used along with direct
calculations~\cite{preprint-2009}.  In unsymmetric case this method
allows  to obtain majorants for $C_{02}$, arbitrarily close to
$C_E$, provided that  the computer used is of sufficient power. The
main content of the preprint~\cite{preprint-2009} was published in 2011,
2012 in the form of the papers \cite{NaCh-Dokl,NaCh}. In these
papers, the following bound was proved, $C_{02}<0.4215$.

 In 2015 we obtained the bound
%
\begin{equation}
\label{953} C_{02} \le 0.4099539, \end {equation} by applying the
same approach as in \cite{preprint-2009,NaCh-Dokl,NaCh}, with the only difference that this time a
supercomputer was used instead of an ordinary PC. We announced bound \eqref{953}
in \cite{NChZ-2016}, but for a number of reasons, delayed publishing the proof,
and do it just now. While the present work being in preparation, we have detected
a small inaccuracy in the calculations, namely, bound \eqref{953} must be
increased by $10^{-7}$. Thus the following statement is true.
\begin{thm}\label{th-2} The bound %
\begin{equation}\label{954} C_{02} \le 0.409954 \end {equation} holds.
\end{thm} Meanwhile, in 2016 J.~Schulz~\cite{Shulz} obtained the unimprovable
result: if the symmetry condition is violated, $C_{02}=C_E$.
As it should be expected, J.~Schulz's proof turned out to be very
long and complicated. It should be said that methods based on the use of
computers, and analytical methods complement each other. The former ones cannot
lead to a final result, but they do not require so much effort. On the other
hand, they allow us to predict the exact result, and thus facilitate theoretical
research. 
\section{Shortly about the proof of Theorem
\ref{th-2}} \label{sect2} 
\subsection{Some notations. On the choice
of the left boundary of the interval for $p$} Let $X,\,X_1,\,
X_2,\ldots,\,X_n$ be a sequence of independent random
variables with the same distribution: %
\begin{equation}\label{Bernoulli}{
\bf P}(X\!=\!1)\!=\!p,\quad {\bf P}(X\!=\!0)=q=1-p.
\end{equation} In what follows we use the
following notations,
%
\begin{align}
&\!F_{n,p}(x)\!=\!{\bf P} \Biggl(\sum\limits
_{i=1}^nX_i<x
\Biggr),\quad G_{n,p}(x)\!=\!\varPhi \biggl({{x-np}\over{\sqrt{npq}}}
\biggr),
\nonumber
\\
&\Delta_n(p)\!=\!\sup\limits
_{x\in \mathbb
R} |F_{n,p}(x)-G_{n,p}(x)
|,\quad \varrho(p)\!=\!\frac{{\bf
E}|X-p|^3}{({\bf E}(X-p)^2)^{3/2}}\!=\!{{p^2+q^2}\over\sqrt{pq}},
\nonumber
\\
&T_n(p)\!=\!{\Delta_n(p)\sqrt{n}\over{\varrho(p)}},\quad {\cal E}(p)=
\frac{2-p}{3\sqrt{2\pi} \,  [p^2 +
(1-p)^2 ]}.\label{Tnp}
\end{align} Obviously,
%
\begin{equation}
\label{C02=sup}C_{02}= \sup\limits
_{n\geq 1}\sup\limits
_{p\in(0,0.5]}T_n(p).
\end{equation}
 In this paper we solve, in particular,  the
problem of computing the sequence
$T(n)=\sup\limits_{p\in(0,0.5)}T_n(p)$ for all~$n$ such that $1\le n
\le N_0$. Here and in what follows,
\begin{align*}
N_0=5\cdot10^5.
\end{align*}

Note that for fixed $ n $ and $ p $, the quantity
$\sup\limits_{x\in\mathbb R}  \llvert  F_{n, p} (x) -G_{n, p} (x)
 \rrvert $ is achieved  at some  discontinuity point of the function $
F_{n, p} (x)$ (see Lemma \ref{lem-D+-}). We consider distribution
functions that are continuous from the left. Consequently,
%
\begin{equation}
\label{Deltanp}\Delta_n(p)=\max_{0\le i\le
n}
\Delta_{n,i}(p),
\end{equation} where $i$ are integers, $
\Delta_{n,i}(p)=
 \{ \llvert F_{n,p}(i)-G_{n,p}(i) \rrvert ,\, \llvert F_{n,p}(i+1)
-G_{n,p}(i) \rrvert  \}$.

Note also that we can vary the parameter $ p $ in a narrower
interval than $ [0,0.5] $,  namely, in
\begin{align*}
I: = [0.1689,0.5]. %
\end{align*}  This conclusion follows from the next
statement.

\begin{lemma}\label{lem-1-ZNC}If $0<p\le0.1689$, then for
all \mbox{$n\!\ge\!1$},
%
\begin{equation}
\label{p=0}T_n(p)\!<\!0.4096.
\end{equation}
\end{lemma}

Lemma \ref{lem-1-ZNC}  is proved in Section~\ref{sect3} with the
help of some modification of the Berry\,--\,Esseen inequality (with
numerical constants) obtained in
\cite{KorShv-Obozr-2010-2,KorShv-Obozr-2010}.

\begin{remark}
By the same method that is used to prove inequality~\eqref{p=0}, the
estimate \mbox{$T_n(p)\le 0.369$} is found in \cite{NaCh} in the
case $0<p<0.02$ ($n\ge1$) (see the proof of~(1.37) in \cite{NaCh}),
where an earlier estimate of V.~Korolev and I.~Shevtsova
\cite{KorShv} is used, instead of
\cite{KorShv-Obozr-2010-2,KorShv-Obozr-2010}. Note that the use of
modified inequalities of the  Berry\,--\,Esseen type, obtained in
\cite{KorShv-Obozr-2010-2,KorShv-Obozr-2010,KorShv}, is not
necessary for obtaining estimates of~$T_n (p) $ in the case when~$ p
$ are close to~0.

An alternative approach, using Poisson approximation, is proposed in
the pre\-print~\cite{preprint-2009}. Let us explain the essence of
this method.

An alternative bound is found in the domain
$\{(p,n):\,0.258\leq\lambda\leq 6,\,n\geq 200\}$, where
$\lambda=np$. Under these conditions, we have  $p\leq 0.03$, i.e.
$p$ are small enough. Consequently, the error arising under
replacement of the binomial distribution by Poisson distribution
$\varPi_\lambda$ with the parameter $\lambda$ is small.

Next, the distance  $d(\varPi_\lambda,G_\lambda)$ \xch{}{in the sense of uniform metric}  between $\varPi_\lambda$ and normal distribution
$G_\lambda$ with the mean $\lambda$ and the variance $\lambda$ is
estimated, where\querymark{Q1} \xch{\mbox{$d(U,V)=\sup\limits_{x\in\mathbb R}|U(x)-V(x)|$}}{$d(U,V)=\sup\limits_{x\in\mathbb R}|U(x)-V(x)|$} for
any distribution functions $U(x)$ and $V(x)$. Then the  estimate of
the distance between $ G_\lambda $ and the normal distribution $
G_{n, p} $ with the mean $ \lambda $ and variance $ npq $ is
deduced. Summing the obtained estimates, we  arrive at an estimate
for the distance between the original binomial distribution and $ G_
{n, p} $. As a result, in \cite[Lemma~7.8, Theorem~7.2]
{preprint-2009} we derive the estimate $ T_n (p) <0.3607 $, which is
valid for all points $ (p, n) $ in the indicated domain.
\end{remark}

\subsection{On calculations}
\label{subsect-2.2}

Define
\begin{align*}
C_{02}(N)=\max_{1\le n\le N}\sup_{p\in(0,0.5]}T_n(p),
\quad \overline C_{02}(N)=\sup_{n\ge
N}\sup
_{p\in(0,0.5]}T_n(p).
\end{align*} Obviously,
$C_{02}=\max\{C_{02}(N),\overline C_{02}(N+1)\}$ for every $N\ge1$.

It was proved in \cite{NaCh} that $\overline C_{02}(200)<0.4215$. By
that time it was shown with the help of a computer (see the preprint
\cite{N-M-Ch-prep}) that $C_{02}(200)<0.4096$, i.e.
%
\begin{equation}
\label{C_{02}(200)} C_{02}({200})<C_E,
\end{equation} and thus, $C_{02}<0.4215$ for all
$n\ge1$.

Some words about bound \eqref{C_{02}(200)}. By \eqref{C02=sup}, to
get $C_{02}(N)$ it is enough to calculate
$T(n)=\sup\limits_{p\in(0,0.5]}T_n(p)$ for every $1\le n\le N$, and
then find $\max\limits_{1\le n\le N}T(n)$. The
calculation of $T(n)$ is reduced to two problems. The first problem is to
calculate $\max\limits_{p_j\in S}T_n(p_j)$, where $S$ is a grid on
$(0,0.5]$, and the second one is to estimate $T_n(p)$ in intermediate
points $p$. Both problems were solved in \cite{N-M-Ch-prep} for
$1\le n\le200$.

It should be noted here that, according to the method, the quantity
$C_{02}(N)$ is calculated (with some accuracy), and $\overline
C_{02}(N)$ is estimated from above. In both cases, a computer is
required. The power of an ordinary PC is sufficient for calculating
majorants for $\overline C_{02}(N)$ whereas to calculate $
C_{02}(N)$  a supercomputer is needed  if $N$ is sufficiently large.
Moreover, an additional investigation of the interpolation type is
required for the convincing conclusion from computer calculations of
$ C_{02}(N)$. In our paper, Theorem~\ref{th-3} plays this role.

Denote by symbol   $S$ the uniform grid on $I$ with the step
   $h=10^{-12}$. The values of $ T_n (p_j) $ for all $ p_j \in S$ and $ 1 \le n \le N_0 $ were calculated on a supercomputer.
   \vskip2mm

\par\noindent\textbf{The result of the calculations}. {\it For all $1\le n\le
N_0$},
%
\begin{equation}
\label{<0.4-S}\max_{p_j\in
S}T_n(p_j)=T_{N_1}(
\overline p)=0.40973212897643\ldots<0.40973213.
\end{equation}

The counting algorithm is a triple loop: a loop with respect to the
parameter $ i $ (see~\eqref{Deltanp}) is nested in a loop with
respect to the parameter $ p $, which in turn is nested in the loop
with respect to the parameter~$ n $.

With the growth of $n$, the computation time increased rapidly. For
example, for $2000 \le n \le 2100$ calculations took more than 3
hours on a computer with processor Core2Due E6400. For $2101\le n\le
N_0$ calculations were carried out on the supercomputer Blue Gene/P.

It follows from \cite[Corollary~7]{NChZ-2016} that for $ n> 200 $ in
the loop with respect to $ i $, one can  take not all values of $ i
$ from 0 to $ n $, but only those, which satisfy the inequality
\begin{align*}
np - (\nu + 1) \sqrt{npq} \! \le \! i \le np + \nu \sqrt{npq}, %
\end{align*} where \mbox{$ \nu \! = \! \sqrt{3+ \sqrt{6}} $}. This
led to a significant reduction of computation time. We give
information about the 
computer time (without waiting for the
queue) in 
Table 1.

\begin{table}[h]
\caption{Dependence of computer time  on $n$ (supercomputer Blue
Gene/P)}\label{tabl-time}
\begin{tabular}{|c|c|c|c|c|c|}
\hline
 $n\in[N_1,N_2]$: & $[10000,11024]$ & $[30000,50000]$ & $[300000,320000]$ & $[490000,N_0]$\\
\hline
computer time: &  3 min & 2 hrs + 5 min & 4 hrs + 50 min & 7 hrs \\
\hline
\end{tabular}
\end{table}

  Calculations were carried out
on the supercomputer Blue Gene/P of the Computational Mathematics
and Cybernetics Faculty of Lomonosov Moscow State University. After
some changes in the algorithm, the calculations for $n$ such that
$490000\le n\le N_0 $, were also performed on the CC FEB RAS
Computing Cluster \cite{DataCenter}. The corresponding computer time
was 6 hours and 40 minutes.

 The
program is written in C+MPI and registered \cite{Zol}.

\subsection{Interpolation type results}

Let $p^\ast\in(0,0.5)$. Consider a uniform grid on $[p^\ast,0.5]$
with a step $h$. The following statement allows to estimate the
value of the function $\frac{1}{\varrho(p)}\,\Delta_{n,k}(p)$ at an
arbitrary point from the interval $[p^\ast,0.5]$ via the value of
this function at the nearest grid node and $h$.

Denote
%
\begin{equation}
\label{c1c2c3} c_1\!=\!0.516, \quad c_2\!=\!0.121,\quad
c_3\!=\!0.271.
\end{equation}

\begin{thm}\label{th-3}  Let $0<p^\ast<p\le0.5$,
$p^{\,\prime}$ be a node of a grid with a step $h$ on the interval
$[p^\ast,0.5]$, closest to $p$. Then for all $n\ge1$ and $0\le k\le
n$,
\begin{align*}
\bigg|\frac{1}{\varrho(p)}\,\Delta_{n,k}(p)-\frac{1}{\varrho(p^\prime)}\,
\Delta_{n,k}\bigl(p^\prime\bigr) \bigg| \le\frac{h}{2}\,L
\bigl(p^\ast\bigr),
\end{align*}
where
%
\begin{equation}
\label{L(p)-new}L(p)\!=\!\frac{1}{(1-2pq)\sqrt{pq}} \biggl(\frac{c_1}{p}+c_2+c_3
\,\frac{(1-2p)(1+2pq)}{1-2pq} \biggr).
\end{equation}
\end{thm}

The next statement follows from Theorem \ref{th-3}. Note that
without it the proof of Theorem~\ref{th-2} would be incomplete.

\begin{cor}\label{th-1.2}  If
  $p\in I$, and $p^\prime$ is a node of the grid $S$, closest to $p$, then for all $1\le
  n\le N_0$,
\begin{align*}
\big|T_n(p)-T_n\bigl(p^\prime\bigr)\big|\le 4.6
\cdot10^{-9}.
\end{align*}
\end{cor}

\begin{proof} It follows
from Theorem~\ref{th-3} that for  $0\le k\le n\le N_0$,
%
\begin{equation}
\label{<L(0.1689)} \bigg|\frac{\sqrt{n}}{\varrho(p)}\,\Delta_{n,k}(p)-\frac{\sqrt{n}}{\varrho(p^\prime)}
\,\Delta_{n,k} \bigl(p^\prime \bigr) \bigg|\le \sqrt{N_0}
\,\frac{1}{2}\,10^{-12}\,L(0.1689).
\end{equation} Since
 $L(0.1689)<12.98$, the right-hand side of
inequality~\eqref{<L(0.1689)}  is majorized by the
number~$4.6\cdot10^{-9}$. This implies the statement of
Corollary~\ref{th-1.2}.\end{proof}

\subsection{On the proof of Theorem~\ref{th-2}}

It follows from \eqref{<0.4-S}, Corollary  \ref{th-1.2} and
Lemma~\ref{lem-1-ZNC}
 that for all $1\le n\le N_0$ and $p\in(0,0.5]$, the
following inequality holds, $T_n(p)<0.4097321346<C_E$ (for details,
see~\eqref{Knp<0.41}). It is easy to verify that this inequality is
true for $ p \in (0.5,1) $ as well. Hence, inequality \eqref{954}
implies Theorem~\ref{th-2}.

\subsection{About structure of the paper}

The structure of  the paper is as follows. The proof of  Theorem
\ref{th-3}, the main analytical result  of the paper, is given in
Section~\ref{sect3-}. The proof consists of 12 lemmas.

In Section \ref{sect3},  Theorem~\ref{th-2} is proved. The section
consists of three subsections. In the first one, the formulation of
Theorem~1.1~\cite{NaCh} is given. Several corollaries from the latter
are also deduced here. The second subsection  discusses the
connection between the result of K. Neammanee \cite{Neamm}, who
refined and generalized
 Uspensky's estimate~\cite{Uspen}, and  the problem of
estimating $C_{02}$. It is shown that one can obtain from the result
of K.~Neammanee the same estimate for $C_{02}$ as ours, but for a
much larger $N$. This means that calculating $C_{02}(N)$  requires
much more computing time if to use Neammanee's estimate.

In the third subsection, we give, in particular, the proof of
Lemma~\ref{lem-1-ZNC}.

\section{Proof of Theorem \ref{th-3}}
\label{sect3-}

We need the following statement, which we give without proof.

\begin{lemma}\label{lem-D+-} Let $G(x)$ be a distribution
function with a finite number of discontinuity points, and $G_0(x)$
a continuous distribution function. Denote $\delta(x)=G(x)-G_0(x)$.
There exists a discontinuity point $x_0$ of $G(x)$ such that the
magnitude $\sup\limits_x|\delta(x)|$ is attained in the following
sense: if $G$ is continuous from the left, then
$\sup\limits_x|\delta(x)|=\max \{\delta(x_0+),\,-\delta(x_0) \}$,
and if $G$ is continuous from the right, then
$\sup\limits_x|\delta(x)|=\max \{\delta(x_0),\,-\delta(x_0-) \}$.
\end{lemma}

Define $f(t)={\bf E}e^{it(X-p)}\equiv qe^{-itp}+pe^{itq}$.

\begin{lemma} \label{lem-2-ZNC} For all $t\in\mathbb R$,
\begin{align*}
|f(t)|\le \exp \biggl\{-2pq\,\sin^2 \frac{t}{2} \biggr\}.
\end{align*}
\end{lemma}

\begin{proof} Taking into account the difference in the notations, we obtain the statement of Lemma \ref{lem-2-ZNC} from~\cite[Lemma 8]{NaCh}.\end{proof}

Further, we will use  the following notations:
\begin{align*}
\sigma=\sqrt{npq},\quad\beta_3(p)={\bf E} {|X-p|^3},
\end{align*} $Y$ is a
standard normal random variable. Note that
$\varrho(p)=\frac{\beta_3(p)}{(pq)^{3/2}}$.

\begin{lemma}\label{lem_f^n-e^n-10} The following bound is
true for all $n\ge2$,
\begin{equation*}
\int_{|t|\le\pi}|f^n(t)-e^{-npqt^2/2}|
\,dt<\frac{1}{\sigma^2}\, \biggl(f(p,n)+\pi \sigma^2e^{-\sigma^2}+
\frac{4}{\pi}\,e^{-\pi^2\sigma^2/8} \biggr)\xch{,}{.}
\end{equation*}
where
\begin{align*}
f(p,n)= \bigl(p^2+q^2 \bigr)\,\frac{\pi^4}{96}\,
\biggl(\frac{n}{n-1} \biggr)^2+ \frac{3\pi^5\sqrt{\pi
pq}}{2^{10}\sqrt{n}}\, \biggl(
\frac{n}{n-1} \biggr)^{5/2}.
\end{align*}
\end{lemma}

\begin{proof}  Using the equalities $e^{-pqt^2/2}={\bf
E}e^{it\sqrt{pq}\,Y}$, ${\bf E}(X-p)^j={\bf
E} (Y\sqrt{pq}\, )^j$, $j=0,1,2$, and the Taylor formula, we
get
%
\begin{multline}
\label{t^3+t^4} |f(t)-e^{-pqt^2/2}|\!=\! \Bigg|{\bf E} \Biggl[\sum
_{j=1}^2\frac{ (it(X-p) )^j}{j!} +\frac{ (it(X-p) )^3}{2}\!\int
_0^1\!\!(1-\theta)^2
\,e^{it\theta(X-p)} \,d\theta \Biggr]
\\
-{\bf E} \Biggl[ \sum_{j=1}^3
\frac{1}{j!}\, (it\sqrt{pq}\,Y )^j +\frac{ (it\sqrt{pq}\,Y )^4}{3!} \int
_0^1(1-\theta)^3\, e^{it\theta\sqrt{pq}\,Y}
\,d\theta \Biggr] \Bigg|
\\
= \Bigg|{\bf E} \Biggl[\frac{ (it(X-p) )^3}{2}\int_0^1(1-
\theta)^2\,e^{it\theta(X-p)} \,d\theta
\\
- \frac{ (it\sqrt{pq}\,Y )^4}{3!} \int_0^1(1-
\theta)^3\, e^{it\theta\sqrt{pq}\,Y} \,d\theta \Biggr] \Bigg| \le\frac{|t|^3}{6}
\, \beta_3(p)+\frac{t^4}{8}\,(pq)^2.
\end{multline}

Since for $|x|\le\frac{\pi}{4}$ the inequality $|\sin
x|\ge\frac{2\sqrt{2}\,|x|}{\pi}$ is fulfilled, then with the help of
Lemma~\ref{lem-2-ZNC} we arrive at the following bound for
$|t|\le\pi/2$,
\begin{equation*}
\nonumber
|f(t)|\le\exp \bigl\{-2pq\sin^2(t/2) \bigr\} \le\exp
\biggl\{- \frac{4t^2pq}{\pi^2} \biggr\}.
\end{equation*} Then,
taking into account the elementary equality $ a^n-b^n = (a-b)
\sum\limits_{j=0}^{n-1}a^jb^{n-1-j} $ and the estimate
\eqref{t^3+t^4}, we obtain for $ | t |\le \pi / 2 $ that
\begin{multline*}
|f^n(t)-e^{-npqt^2/2}|\le |f(t)-e^{-pqt^2/2}|\sum
_{j=0}^{n-1}|f(t)|^je^{-(n-1-j)t^2pq/2}
\le
\\
\le \biggl(\frac{|t|^3}{6}\, \beta_3(p)+\frac{t^4}{8}\,
(pq)^{2} \biggr)\,\sum_{j=0}^{n-1}
\exp \bigl\{ \bigl[j\, \bigl(1-8/\pi^2 \bigr)-(n-1)
\bigr]t^2pq/2 \bigr\}\le
\\
\le \biggl(\frac{|t|^3}{6}\, \beta_3(p)+\frac{t^4}{8}\,
(pq)^{2} \biggr)\,n \,\exp \biggl\{-\frac{4(n-1)t^2pq}{\pi^2} \biggr\}.
\end{multline*}
Using the well-known formulas ${\bf E}|Y|^3=\frac{4}{\sqrt{2\pi}}$
and ${\bf E}Y^4=3$, we deduce from the previous inequality that for
$n\ge2$,
%
\begin{multline}
\label{f(p,n>2)}\int\limits
_{|t|\le\pi/2}|f^n(t)-e^{-npqt^2/2}|\,dt\le n\sqrt{2
\pi} \, \biggl(\frac{\beta_3(p)}{6m^2}\,{\bf E}|Y|^3+\frac{(pq)^2}{8m^{5/2}}\,{
\bf E}Y^4 \biggr) \bigg|_{m=\frac{8(n-1)pq}{\pi^2}}
\\
= n \biggl(\frac{\pi^4\varrho(p)}{96\sqrt{pq}\,(n-1)^2} +\frac{3\pi^5\sqrt{\pi}}{2^{10}\sqrt{pq}\,(n-1)^{5/2}} \biggr) =\frac{f(p,n)}{\sigma^2}.
\end{multline}

Applying Lemma \ref{lem-2-ZNC} again, we get
%
\begin{equation}
\label{int_x>pi/6}\int\limits
_{\pi/2\le |t|\le\pi}|f^n(t)|\,dt\le 2\int_{\pi/2}^{\pi}
e^{-2\sigma^2\,\sin^2
(t/2)}dt<\pi\,e^{-\sigma^2}.
\end{equation} Moreover, by virtue of
the known inequality
%
\begin{equation}
\label{norm-0}\int_c^\infty e^{-t^2/2}\,dt
\le \frac{1}{c}\,e^{-c^2/2},
\end{equation} which holds
for every $c>0$, we have
%
\begin{equation}
\label{e^{-nt}}\int_{|t|\ge
\pi/2}e^{-\sigma^2t^2/2}\,dt \le
\frac{4}{\pi
\sigma^2}\,e^{-\sigma^2\pi^2/8}.
\end{equation} Collecting the
estimates \eqref{f(p,n>2)}--\eqref{e^{-nt}}, we obtain the
statement of Lemma \ref{lem_f^n-e^n-10}.
\end{proof}

Denote
\begin{align*}
P_n(k)=C_n^kp^kq^{n-k},
\quad \delta_n(k,p)=P_n(k)-\frac{1}{\sqrt{npq}}\,\varphi
\biggl(\frac{k-np}{\sqrt{npq}} \biggr).
\end{align*}

\begin{lemma}\label{lem_loc} For every $n\ge1$ and $0\le
k\le n$ the following bound holds,
%
\begin{equation}
\label{loc-}|\delta_n(k,p)|<\min \biggl\{\frac{1}{\sigma\sqrt{2e}},
\frac{c_1}{\sigma^2} \biggr\},
\end{equation}
where $c_1$ is defined in \eqref{c1c2c3}.
\end{lemma}

\begin{proof}
  It was proved in \cite{Herzog} that
$P_n(k)\le\frac{1}{\sqrt{2enpq}}$. Moreover,
$\frac{1}{\sqrt{npq}}\,\varphi (\frac{k-np}{\sqrt{npq}} )\le\frac{1}{\sqrt{2\pi
npq}}$. Hence,
%
\begin{equation}
\label{P-fi-2}|\delta_n(k,p)|\le\frac{1}{\sqrt{2e
npq}}=
\frac{1}{\sigma\sqrt{2e }}.
\end{equation}

Let us find another bound for $\delta_n(k,p)$. Let $\sigma>1$. Then
$n>\frac{1}{pq}\ge4$, i.e. $n\ge5$.

By the inversion formula for integer random variables,
\begin{align*}
P_n(k)=\frac{1}{2\pi}\int_{-\pi}^\pi
\bigl(q+e^{it}p\bigr)^n\,e^{-itk}\,dt=
\frac{1}{2\pi}\int_{-\pi}^\pi f^n(t)
\,e^{-it(k-np)}\,dt.
\end{align*} Moreover,
by the inversion formula for densities,
\begin{align*}
\frac{1}{\sigma}\,\varphi \biggl(\frac{x-\mu}{\sigma} \biggr)= \frac{1}{2\pi}
\int_{-\infty}^\infty e^{-t^2\sigma^2/2-it(x-\mu)}\,dt.
\end{align*} Consequently,
%
\begin{equation}
\label{J1-J2}\delta_n(k,p)=\frac{1}{2\pi} \,
(J_1-J_2 ),
\end{equation} where
\begin{align*}
J_1=\int_{-\pi}^\pi
\bigl[f^n(t)-e^{-\sigma^2t^2/2} \bigr]\,e^{-it(k-np)}\,dt,\quad
J_2=\int_{|t|\ge\pi} e^{-\sigma^2t^2/2}\,e^{-it(k-np)}
\,dt.
\end{align*} Note
that the function $f(p,n)$ from  Lemma \ref{lem_f^n-e^n-10}
decreases in $n$. Hence, $f(p,n)\le f(p,5)$.   It is not hard to
verify that $\max\limits_{p\in[0,1]}f(p,5)<1.707$. Thus, for
$\sigma>1$,
\begin{align*}
|J_1|\le \frac{1}{\sigma^2} \biggl(1.707+\frac{\pi}{e}+
\frac{4}{\pi}\,e^{-\pi^2/8} \biggr)<\frac{3.234}{\sigma^2}.
\end{align*} Using
inequality \eqref{norm-0}, we get the estimate
\begin{align*}
|J_2|\le \frac{2}{\pi\sigma^2}\,e^{-\pi^2\sigma^2/2}<\frac{0.005}{\sigma^2}.
\end{align*}
Thus, we get from~\eqref{J1-J2} that for $\sigma>1$,
%
\begin{equation}
\label{P-fi} |\delta_n(k,p)| \le \frac{3.24}{2\pi\sigma^2}<
\frac{0.516}{\sigma^2}.
\end{equation}

Since $\frac{1}{\sigma\sqrt{2e}}\le \frac{c_1}{\sigma^2}$ for
$0<\sigma\le c_1\sqrt{2e}=1.203\ldots\;>1$, the statement of Lemma
\ref{lem_loc} follows from \eqref{P-fi-2} and \eqref{P-fi}.
\end{proof}

\begin{lemma}\label{lem_d/dpG} The following equality holds,
%
\begin{equation}
\label{d/dpG}\frac{\partial}{\partial
p}\,G_{n,p}(x)=-\frac{x(1-2p)+np}{2pq\sqrt{npq}}\,
\varphi \biggl(\frac{x-np}{\sqrt{npq}} \biggr).
\end{equation}
\end{lemma}

\begin{proof} We have
\begin{align*}
&\frac{d}{d p}p^{-1/2}(1-p)^{-1/2}=-\frac{q-p}{2pq\sqrt{pq}},
\\
&\frac{d}{d
p}p^{1/2}(1-p)^{-1/2}=\frac{1}{2}p^{-1/2}(1-p)^{-1/2}+
\frac{1}{2}p^{1/2}(1-p)^{-3/2} =\frac{1}{2q\sqrt{pq}}.
\end{align*} Hence,
\begin{align*}
\frac{\partial}{\partial p}\,\frac{x-np}{\sqrt{npq}}=-\frac{x(q-p)}{2pq\sqrt{npq}}- \frac{\sqrt{n}}{2q\sqrt{pq}}=-
\frac{x(q-p)+np}{2pq\sqrt{npq}},
\end{align*} and
we arrive at \eqref{d/dpG}. \end{proof}

\begin{lemma}\label{lem_loc-2}  For all  $n\ge1$ and $0\le
k\le n$ the following bound holds,
\begin{align*}
\bigg|\frac{\partial}{\partial p}\,F_{n,p}(k+1)- \frac{\partial}{\partial p}
\,G_{n,p}(k) \bigg|\le L_1(p)\equiv\frac{1}{pq} \biggl(
\frac{c_1}{q}+ c_2 \biggr).
\end{align*}
\end{lemma}

\begin{proof}  It is shown in \cite{Shmet} that
\begin{align*}
\frac{\partial}{\partial
p}\,F_{n,p}(k+1)=-nC_{n-1}^kp^kq^{n-1-k}=-
\frac{n-k}{q}P_n(k). %
\end{align*}
By Lemma \ref{lem_loc},
%
\begin{equation}
\label{P-fi-3}\frac{n-k}{q}\; \bigg|P_n(k)- \frac{1}{\sigma}\,
\varphi \biggl(\frac{k-np}{\sigma} \biggr) \bigg|\le \frac{n\,c_1}{q\sigma^2}=
\frac{c_1}{pq^2}.
\end{equation} In turn,
it follows from Lemma \ref{lem_d/dpG} that
%
\begin{multline}
\label{fi-fi}\frac{n-k}{q\sigma} \,\varphi \biggl(\frac{k-np}{\sigma} \biggr)+
\frac{\partial}{\partial
p}\,G_{n,p}(k)
\\
= \biggl(\frac{n-k}{q\sigma} - \frac{k(1-2p)+np}{2pq\sigma} \biggr)\, \,\varphi \biggl(
\frac{k-np}{\sigma} \biggr)=- \frac{k-np}{2pq\sigma}\,\varphi \biggl(\frac{k-np}{\sigma}
\biggr).
\end{multline}
Since
\begin{multline*}
\frac{\partial}{\partial p}\,F_{n,p}(k+1)- \frac{\partial}{\partial
p}
\,G_{n,p}(k)= -\frac{n-k}{q}\; \biggl[P_n(k)-
\frac{1}{\sigma}\,\varphi \biggl(\frac{k-np}{\sigma} \biggr) \biggr]
\\
- \biggl[\frac{n-k}{q\sigma}\, \varphi \biggl(\frac{k-np}{\sigma} \biggr)+
\frac{\partial}{\partial p}\,G_{n,p}(k) \biggr]
\end{multline*} and
$\max\limits_{x}|x|\varphi(x)=\frac{1}{\sqrt{2\pi e}}<0.242$, the
statement of the lemma follows from~\eqref{P-fi-3}
and~\eqref{fi-fi}.
\end{proof}

\begin{lemma}
\label{lem_loc-3} For all  $n\ge1$ and $0\le k\le n$ the following
bound holds,
\begin{align*}
\bigg|\frac{\partial}{\partial p}\,F_{n,p}(k)- \frac{\partial}{\partial p}
\,G_{n,p}(k) \bigg|\le L_2(p)\equiv\frac{1}{pq} \biggl(
\frac{c_1}{p}+ c_2 \biggr),
\end{align*} where $c_1$,
$c_2$ are from \eqref{c1c2c3}.
\end{lemma}
\begin{proof} Similarly to the proof of Lemma \ref{lem_loc-2} we
obtain
\begin{align}
&\frac{\partial}{\partial
p}\,F_{n,p}(k)=-nC_{n-1}^{k-1}p^{k-1}q^{n-k}=-
\frac{k}{p}P_n(k),
\nonumber\\
%
\label{P-fi-1}&\frac{k}{p}\; \bigg|P_n(k)- \frac{1}{\sigma}\,
\varphi \biggl(\frac{k-np}{\sigma} \biggr) \bigg|\le \frac{k\,c_1}{p\sigma^2}\le
\frac{c_1}{p^2q}.
\end{align}
Hence,
\begin{align*}
\frac{\partial}{\partial p}\,F_{n,p}(k)- \frac{\partial}{\partial
p}\,G_{n,p}(k)=
-\frac{k}{p}\; \biggl[P_n(k)- \frac{1}{\sigma}\,\varphi
\biggl(\frac{k-np}{\sigma} \biggr) \biggr]- \frac{k-np}{2pq\sigma}\varphi \biggl(
\frac{k-np}{\sigma} \biggr).
\end{align*}
Since the last summand on the right-hand side of the equality is
less than $\frac{0.121}{pq}$, then by using~\eqref{P-fi-1} we get
the statement of the lemma. \end{proof}

\begin{lemma}\label{lem-2.8} For every $0<p<0.5$,
%
\begin{equation}
\frac{d}{dp}\frac{1}{\varrho(p)}=\frac{1}{2}\,A(p):=
\frac{1}{2}\,\frac{(1-2p)(1+2pq)}{\sqrt{pq}(1-2pq)^2}.\label{(2.8.0)}
\end{equation}
\end{lemma}
\begin{proof}   The lemma follows from the equalities:
\begin{align*}
&\frac{d}{dp}\,\frac{1}{\varrho(p)}= \frac{d}{dx}\,\frac{x}{1\,{-}\,2x^2}
\bigg|_{x=\sqrt{pq}}\times\frac{d}{dp}\sqrt{p(1\,{-}\,p)},\;\;\frac{d}{dp}
\sqrt{p(1\,{-}\,p)}=\frac{1-2p}{2\sqrt{pq}},
\\
& \frac{d}{dx}\,\frac{x}{1-2x^2}=\frac{1}{1-2x^2} +
\frac{4x^2}{(1-x^2)^2}=\frac{1+2x^2}{(1-2x^2)^2}.\qedhere
\end{align*}
\end{proof}

\begin{lemma}\label{lem-2.9} The function $A(p)$ decreases
on the interval $(0,0.5)$.\end{lemma}

\begin{proof} Denote $x=x(p)=p(1-p)$, $A_1(t)=
\frac{\sqrt{1-4t}\,(1+2t)}{\sqrt{t}\,(1-2t)^2}$. Taking into account
the equality $1-2p=\sqrt{1-4pq}$, we obtain $A(p)=A_1(x)$.

Since $x(p)$ increases for $0<p<0.5$, it remains to prove the
decrease of the function $A_1(x)$ for $0< x<0.25$. We have
\begin{align*}
\frac{d}{dx} \,\ln{A_1(x)}=\frac{-2}{1-4x}+
\frac{2}{1+2x}-\frac{1}{2x}+\frac{4}{1-2x} =-\frac{32x^3+36x^2-12x+1 }{2x(1-4x)(1-4x^2)}.
\end{align*} On the interval
$[0,0.25]$ the polynomial $A_2(x)\equiv32x^3+36x^2-12x+1$ has the
single minimum point $x_1=\frac{-3+\sqrt{17}}{8}=0.140\ldots\;$.
Since $A_2(x_1)=0.11\ldots>0$, we have $\frac{d}{dx}\,\ln A_1(x)<0$
 for $0\le x<0.25$, i.e.  the function $A_1(x)$ decreases on $(0,0.25)$.
 The lemma is proved.\end{proof}

\begin{lemma}\label{lem-L(p)ubyv} The function $L(p)$,
defined in  \eqref{L(p)-new}, decreases on $[0,0.5]$.\end{lemma}

\begin{proof} Taking into account the equality $p^2+q^2=1-2pq$, it is
not difficult to see that
%
\begin{equation}
\label{L(p)-new2}L(p)=\frac{1}{\varrho(p)}\,L_2(p)+c_3\,A(p).
\end{equation}
According to Lemma \ref{lem-2.9}, the function $A(p)$ decreases.
Consequently, it remains to prove that the function
$L_3(p):=\frac{1}{\varrho(p)}\,L_2(p)=\frac{c_1+c_2p}{p\sqrt{pq}(1-2pq)}$
decreases on  $[0,0.5]$. We have
\begin{multline*}
\frac{d}{dp}\,\ln L_3(p)=\frac{c_2}{c_1+c_2p}-
\frac{3}{2p}+\frac{1}{2(1-p)}+\frac{2(1-2p)}{
1-2p+2p^2}
\\
=
\frac{A_3(p)}{2pq(c_1+c_2p)(1-2pq)},
\end{multline*}
where $A_3(p)=-3 c_1 + (14 c_1  - c_2) p - (26 c_1-8c_2) p^2 + (16
c_1 -
 18 c_2) p^3 + 12 c_2 p^4$. Let us prove that
%
\begin{equation}
\label{A3(p)<0} A_3(p)<0,\quad 0<p<0.5.
\end{equation}
We have
\begin{align*}
&A_3^\prime(p)=14 c_1 - c_2 -4 (13
c_1-4c_2) p +6 (8 c_1 - 9 c_2)
p^2 + 48 c_2 p^3,
\\
&A_3^{\prime\prime}(p)= -4 (13 c_1-4c_2) +12
(8 c_1 - 9 c_2) p + 144 c_2 p^2.
\end{align*}
As a result of calculations, we find that the equation
$A_3^{\prime}(p)=0$ has the single root $p_0=0.478287\ldots$ on
$[0,0.5]$. The roots of the equation $A_3^{\prime\prime}(p)=0$ have
the form
\begin{align*}
p_{1,2}=\frac{1}{24c_2}\, \bigl(-8c_1+9c_2
\pm\sqrt{(8c_1-9c_2)^2+16c_2(13c_1-4c_2)}
\, \bigr),
\end{align*}
and are equal to $p_1=-2.6\ldots\;$, $p_2=0.54\ldots\;$
respectively. Hence, $A_3^{\prime\prime}(p)<0$ for $p\in[0,0.5]$.
Thus, the function $A_3(p)$, considered on $[0,0.5]$, takes a
maximum value at the point $p_0$. Since
\mbox{$A_3(p_0)=-0.257\ldots\;$},  inequality~\eqref{A3(p)<0} is
proved. This implies that $L_3(p)$ decreases on
$(0,0.5)$.\end{proof}

Let $f(x)$ be an arbitrary function. Denote by $D^+f(x)$ and
$D^-f(x)$ its right-side and left-side derivatives respectively (if
they exist).
\begin{lemma}\label{lem-f11} Let
$g(x)=\max\{f_{1}(x),\,f_{2}(x)\}$, where  $f_1(x)$ and $f_2(x)$
are functions, differentiable on a finite interval $(a,b)$. Then at
every point $x\in(a,b)$ there exist both one-side derivatives $D^+
g(x)$ and $D^- g(x)$, each of which coincides with either
$f^\prime_{1}(x)$ or $f^\prime_{2}(x)$. \end{lemma}

\begin{proof} Let  $x$ be a point such that
 $f_{1}(x)\neq f_{2}(x)$. Then the function $g$  is differentiable at~$x$, and in this case the statement of the lemma is trivial.

Now let for a point $x\in(a,b)$,
%
\begin{equation}
\label{f1=f2}f_{1}(x)=f_{2}(x).
\end{equation}

First, consider the case $f_1^\prime(x)\neq f_2^\prime(x)$. Let, for
instance, $f_1^\prime(x)> f_2^\prime(x)$. Then there exists $h_0>0$
such that
%
\begin{align}
\label{f1>f2}&f_1(x+h)>f_2(x+h),\quad 0<h\le
h_0,
\\
&\label{f2>f1}f_2(x+h)>f_1(x+h), \quad
-h_0\le h<0.
\end{align}
From differentiability of the functions $f_1$ and $f_2$ it follows
that for $h\to0$,
%
\begin{equation}
\label{fi}f_i(x+h)=f_i(x)+f_i^\prime(x)h+o(h),
\quad i=1,2.
\end{equation} Then using \eqref{f1>f2} we obtain the equality
\begin{align*}
g(x+h)=f_1(x+h)=f_1(x)+f_1^\prime(x)h+o(h),
\quad h>0\xch{,} {.}
\end{align*}
and using \eqref{f2>f1},
\begin{align*}
g(x+h)=f_2(x+h)=f_2(x)+f_2^\prime(x)h+o(h),
\quad h<0.
\end{align*}
Thus, existence of $D^+ g(x)$ and $D^- g(x)$ follows.

Now let
%
\begin{equation}
\label{df1=df2}f_1^\prime(x)= f_2^\prime(x).
\end{equation} It follows from \eqref{f1=f2},
\eqref{fi} and \eqref{df1=df2}  that for $h\to0$,
\begin{align*}
g(x+h)=f_i(x)+f_i^\prime(x)h+o(h),\quad i=1,2.
\end{align*} Hence, $g^\prime(x)=f_1^\prime(x)=f_2^\prime(x)$. The lemma
is proved. \end{proof}

Denote
\begin{align*}
\varrho=\varrho(p),\quad q_i=1-p_i,\quad
\varrho_i=\varrho(p_i)\equiv\frac{\omega(p_i)}{\sqrt{p_iq_i}}.
\end{align*}

\begin{lemma}\label{lem-2.10} Let  $ 0< p_1<p<p_2\le0.5$.
Then for all $n\ge1$ and $\;0\le k\le n$,
%
\begin{equation}
\bigg|\frac{1}{\varrho}\,\Delta_{n,k}(p)- \frac{1}{\varrho_1}\,
\Delta_{n,k}(p_1) \bigg|\le L(p_1)
\,(p-p_1),\quad \;\label{(2.17)}
\end{equation} and
%
\begin{equation}
\bigg|\frac{1}{\varrho}\Delta_{n,k}(p)-\frac{1}{\varrho_2}
\Delta_{n,k}(p_2 ) \bigg|<L(p_1) (p_2-p).\label{(2.10.17)}
\end{equation}\end{lemma}

\begin{proof} Note that $\Delta_{n,k}(p)<0.541$ (see~\cite{arxive-2007}). Consequently,
%
\begin{equation}
\bigg|\frac{1}{\varrho}\,\Delta_{n,k}(p)- \frac{1}{\varrho_1}\,
\Delta_{n,k}(p_1) \bigg|\le \frac{1}{\varrho_1}\,|
\Delta_{n,k}(p)-\Delta_{n,k}(p_1)|+0.541 \biggl(
\frac{1}{\varrho}- \frac{1}{\varrho_1} \biggr)\;.\label{(2.18)}
\end{equation}

It is obvious that $F_{n,p}(k)$ and $G_{n,p}(k)$, considered as
functions of the argument
 $p$, are differentiable. Then, according to Lemma~\ref{lem-f11}, the one-side derivatives of the functions $\Delta_{n,k}(p)$ exist at each
 point
$p\in[0,0.5]$ and coincide with
 $\frac{\partial}{\partial p} (F_{n,p}(k+1)-G_{n,p}(k) )$ or $\frac{\partial}{\partial
 p} (G_{n,p}(k)-F_{n,p}(k) )$.

Taking into account that $L_1(p)\le L_2(p)$ for $0<p\le0.5$, we
obtain from Lemmas~\ref{lem_loc-2} and \ref{lem_loc-3}
%
\begin{multline}
|\Delta_{n,k}(p)-\Delta_{n,k}(p_1)|\le
(p-p_1)\max_{p_1\le s\le
p} |D^+\Delta_{n,k}(s) |
\\
\le(p-p_1)\max_{p_1\le s\le p} L_2(s).\label{(2.10.18)}
\end{multline} The function $L_2(s)$
decreases on
 $(0,\,0.5]$. Hence,
%
\begin{equation}
\max_{p_1\le s\le p} L_2(s)=L_2(p_1).\label{(2.10.19)}
\end{equation} The inequality
%
\begin{equation}
\frac{1}{\varrho_1}\, |\Delta_{n,k}(p)- \Delta_{n,k}(p_1)
|\le\frac{p-p_1}{\varrho_1}\,L_2(p_1)\label{(2.19)}
\end{equation}
follows from  \eqref{(2.10.18)} and \eqref{(2.10.19)}.
    Taking into account Lemmas~\ref{lem-2.8} and \ref{lem-2.9}, we have
%
\begin{equation}
\label{(2.20)}\frac{1}{\varrho}-\frac{1}{\varrho_1}\le (p-p_1)\,
\max\limits
_{p_1<s<p}\frac{d}{ds}\frac{1}{\varrho(s)}<2^{-1}A(p_1)
(p-p_1).
\end{equation} Collecting the estimates \eqref{(2.18)}, \eqref{(2.19)}, \eqref{(2.20)}, we obtain   with the help of
 \eqref{L(p)-new2} that for $0\le p_1<p\le0.5$,
%
\begin{multline}
\label{p1<p} \bigg|\frac{1}{\varrho}\,\Delta_{n,k}(p)-
\frac{1}{\varrho_1}\,\Delta_{n,k}(p_1) \bigg|\le
(p-p_1) \biggl(\frac{1}{\varrho_1}\,L_2(p_1)+0.271
\,A(p_1) \biggr)
\\
= (p-p_1)L(p_1).
\end{multline} Hence, for $0<p<p_2\le0.5$,
%
\begin{equation}
\label{p2-p} \bigg|\frac{1}{\varrho}\Delta_{n,k}(p)-\frac{1}{\varrho_2}
\Delta_{n,k}(p_2 ) \bigg|<(p_2-p)L(p).
\end{equation}

Inequality \eqref{(2.17)} coincides with \eqref{p1<p}, and
inequality  \eqref{(2.10.17)} follows from \eqref{p2-p} and
Lemma~\ref{lem-L(p)ubyv}.
 Lemma~\ref{lem-2.10} is proved.
\end{proof}

\begin{proof}[Proof of Theorem \ref{th-3}] It follows from the definition of
$p^\prime$ that  either $0<p-p^{\,\prime}<h/2$ or
$0<p^{\,\prime}-p<h/2$. In the first case the statement of  the
theorem follows from \eqref{(2.17)} and Lemma~\ref{lem-L(p)ubyv},
and in the second one from \eqref{(2.10.17)} and
Lemma~\ref{lem-L(p)ubyv} again.
\end{proof}

\section{Proof of Theorem \ref{th-2}}
\label{sect3}

\subsection{Theorem 1.1~\cite{NaCh} and some its consequences}
\label{sect2+1}

First we formulate Theorem 1.1  from  \cite{NaCh}. To do this, we
need to enter a rather lot of notations from \cite{NaCh}:
\begin{align*}
&\omega_3(p)=q-p,\quad \omega_4(p)=|q^3+p^3-3pq|,
\quad \omega_5(p)=q^4-p^4,
\\
&\omega_6(p)=q^5+p^5+15(pq)^2,
\\
&K_1(p,n)=\frac{\omega_3(p)}{4\sigma\sqrt{2\pi}(n-1)}\, \biggl(1+\frac{1}{4(n-1)}
\biggr)+ \frac{\omega_4(p)}{12\sigma^2\pi}\, \biggl(\frac{n}{n-1} \biggr)^2
\\
&\hspace*{30mm}+\,\frac{\omega_5(p)}{40\sigma^{3}\sqrt{2\pi}}\, \biggl(\frac{n}{n-1}
\biggr)^{5/2}+ \frac{\omega_6(p)}{90\sigma^{4}\pi}\, \biggl(\frac{n}{n-1}
\biggr)^{3};
\end{align*}
\vspace*{-12pt}
\begin{align*}
\omega(p)&=p^2+q^2,& \zeta(p)&= \biggl(
\frac{\omega(p)}{6} \biggr)^{2/3},\quad e(n,p)=\exp \biggl\{
\frac{1}{24\sigma^{2/3} \zeta^2(p)} \biggr\},
\\
e_5&=0.0277905,& \widetilde{\omega}_5(p)&=p^4+q^4+5!
\,e_5(pq)^{3/2},
\end{align*}
\vspace*{-18pt}
\begin{align*}
&V_6(p)=\omega_3^2(p),\quad
V_7(p)=\omega_3(p)\omega_4(p),\quad
V_8(p)=\frac{2\widetilde{\omega}_5(p)\omega_3(p)}{5!3!}+ \biggl(\frac{\omega_4(p)}{4!}
\biggr)^2,
\\
&V_9(p)=\widetilde{\omega}_5(p)\omega_4(p),
\quad V_{10}(p)=\widetilde{\omega}_5^2(p),\quad
A_k(n)= \biggl(\frac{n}{n-2} \biggr)^{k/2}\,
\frac{n-1}{n},
\end{align*}
\vspace*{-15pt}
\begin{align*}
\begin{array}{lllll}
\gamma_6=\frac{1}{9},&\quad\gamma_7=\frac{5\sqrt{2\pi}}{96},&\quad\gamma_8=24,&\quad
\gamma_9=\frac{7\sqrt{2\pi}}{4!\,16},&\quad\gamma_{10}=\frac{2^6\cdot
3}{(5!)^2},
\\[12pt]
\widetilde\gamma_6=\frac{2}{3},&\quad\widetilde\gamma_7=\frac{7}{8},&\quad\widetilde\gamma_8=\frac{10}{9},
&\quad\widetilde\gamma_9=\frac{11}{8},&\quad\widetilde\gamma_{10}=\frac{5}{3},
\end{array} %
\end{align*}
\vspace*{-15pt}
\begin{equation*}
K_2(p,n)=\frac{1}{\pi\sigma}\sum_{j=1}^5
\frac{\gamma_{j+5}\,A_{j+5}(n)\,V_{j+5}(p)}{\sigma^j}\, \biggl[1 +\frac{\widetilde\gamma_{j+5}\,e(n,p)\,n}{\sigma^{2}\,(n-2)} \biggr]; 
\end{equation*}
\vspace*{-15pt}
\begin{align*}
A_1=5.405,\quad A_2=7.521,\quad A_3=5.233,
\quad \mu=\frac{3\pi^2-16}{\pi^4}, 
\end{align*}
\vspace*{-18pt}
\begin{align*}
\chi(p,n)= \frac{2\zeta(p)}{\sigma^{2/3}}\;\,\text{\rm if}\;\, p\in(0,0.085),\;\;\text{\rm
and}\;\; \chi(p,n)=0\;\,\text{\rm if}\;\, p\in[0.085,0.5],
\end{align*}
\vspace*{-9pt}
\begin{equation*}
\begin{split} K_3(p,n)&=\frac{1}{\pi}\, \biggl\{
\frac{1}{12\sigma^2}+ \biggl(\frac{1}{36}+\frac{\mu}{8} \biggr)\,
\frac{1}{\sigma^4}+ \biggl(\frac{1}{36}\,e^{A_1/6}+
\frac{\mu}{8} \biggr)\, \frac{1}{\sigma^6}+ \frac{5\mu}{24}
\,e^{A_2/6}\, \frac{1}{\sigma^8}
\\
&+\,\frac{1}{3}\,\exp \biggl\{ -\sigma\sqrt{A_1}+
\frac{A_1}{6} \biggr\}+(\pi-2)\mu\exp \biggl\{ -\sigma\sqrt{
A_2}+\frac{ A_2}{6} \biggr\}
\\
&+\,\exp \biggl\{-\sigma\sqrt{A_3}+\frac{A_3}{6} \biggr\}
\frac{1}{4}\,\ln \biggl(\frac{\pi^4 \sigma^2}{4A_3} \biggr)
\\
&+\,\exp \biggl\{-\frac{\sigma^{2/3}}{2\zeta(p)} \biggr\} \biggl[\frac{2\zeta(p)}{\sigma^{2/3}}
+e^{A_3/6}\,\frac{1+\chi(p,n)}{24\,\zeta(p)\,\sigma^{4/3}} \biggr] \biggr\}; \end{split}
\end{equation*}
\vspace*{-6pt}
\begin{equation}
\label{R=K1+}R(p,n)=K_1(p,n)+K_2(p,n)+K_3(p,n).
\end{equation}

\begin{tthm}[{\cite[Theorem 1.1]{NaCh}}]\label{thmA}
Let
%
\begin{equation}
\label{p>0.02} \frac{4}{n}\le p\le0.5,\quad n\ge200.
\end{equation} Then
%
\begin{equation}
\label{Delta-Th} \Delta_n(p)\le \frac{\varrho(p)}{\sqrt{n}}\,{\cal
E}(p)+R(p,n),
\end{equation}  and
the sequence $R_0(p,n):=\frac{\sqrt{n}}{\varrho(p)}\,R(p,n)$ tends
to zero for every $0<p\le0.5$, decreasing in~$n$.
\end{tthm}

Denote
\begin{align*}
E(p,n)={\cal E}(p)+R_0(p,n).
\end{align*} Figure \ref{ris3} shows the mutual location of the
following functions: $E(p, n)$ for $ n = 200 $ and $800 $, $ {\cal
E}(p) $ and $ T_n(p)  |_{n = 50}$. Note that, as a consequence of
the definition of the binomial distribution, the behavior of these
functions is symmetric with respect to $ p = 0.5 $.

\begin{figure}[ht]
\includegraphics{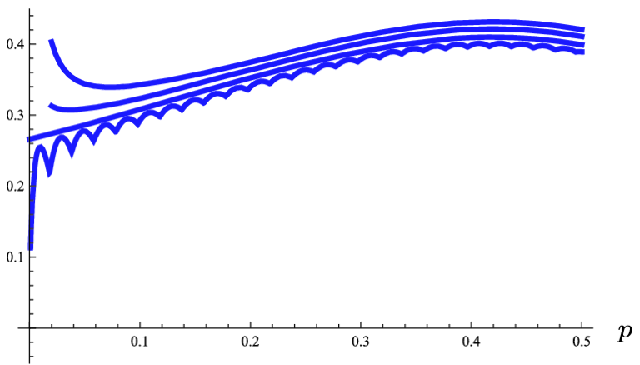}
\caption{%
Graphs of the functions (from top to down): \mbox{$\!E(p,\!200),\,
\!E(p,\!800),\,\!{\cal E}(p),\,\!T_{50}(p)$}}
\label{ris3}
\end{figure}

Recall that $N_0=500000$.

\begin{ccor}\label{corA}
For $p\in[0.1689,0.5]$, and $ n\ge
N_0$,
\begin{align*}
E(p,n)\le E(p,N_0)<0.409954.
\end{align*}
\end{ccor}

\begin{proof}
Since $E(p,n)$ decreases in $n$, we obtain the statement of
Corollary~\ref{corA} by finding the maximal value of $E(p,N_0)$ directly using a computer.
\end{proof}

In order to verify the plausibility of the previous numerical
result, we estimate the function $E(p,N_0)$, making preliminary
estimates of some of the terms that enter into it. This leads to the
following somewhat more coarse inequality.

\begin{CCOR}
For $p\in[0.1689,0.5]$, and $ n\ge
N_0$,
%
\begin{equation}
\label{Epn<}E(p,n)<0.409954153.
\end{equation}
\end{CCOR}

\begin{proof}
Separate the proof of  \eqref{Epn<} into four steps. First we rewrite
 $R_0(p,n)$ in the following form,
\begin{align*}
R_0(p,n)=\frac{K_1(p,n)\sigma}{\omega(p)}+\frac{K_2(p,n)\sigma}{\omega(p)}+
\frac{K_3(p,n)\sigma}{\omega(p)}.
\end{align*}
In each function $\frac{K_i(p,n)\sigma}{\omega(p)}$, $i=1,2,3$, we
will select the principal term, and estimate the remaining ones.

Step 1. Note that for $n\ge N_0$ and $0<a\le3$,
\begin{align*}
&\biggl(\frac{n}{n-1} \biggr)^a\le \biggl(\frac{n}{n-1}
\biggr)^3<e_1:=1.00000601,
\\
& 1+\frac{1}{4(n-1)}
\!<e_2:=1.000000501.
\end{align*} Then
\begin{align*}
\frac{K_1(p,n)\,\sigma}{\omega(p)}=\frac{\omega_4(p)}{12\pi\omega(p)\sigma} \biggl(\frac{n}{n-1}
\biggr)^2 +r_1(p,n),
\end{align*}
where
\begin{align*}
r_1(p,n)<\widetilde r_1(p,n):=\frac{e_1}{\omega(p)} \biggl(
\frac{e_2\,\omega_3(p)}{4\sqrt{2\pi}(n-1)} +\frac{\omega_5(p)}{40\sqrt{2\pi}\sigma^2}+\frac{\omega_6(p)}{90\pi
\sigma^3} \biggr).
\end{align*}
 Using a computer, we get the estimate $\widetilde r_1(p,n)\le \widetilde
r_1(0.1689,N_0)<2.78\cdot10^{-7}$.

Step 2. We have
\begin{align*}
\frac{K_2(p,n)\sigma}{\omega(p)}= \frac{\gamma_{6}\,A_{6}(n)\,V_{6}(p)}{\pi\omega(p)\sigma}+r_2(p,n),
\end{align*}
where
\begin{align*}
r_2(p,n)= \sum_{j=2}^5
\frac{\gamma_{j+5}A_{j+5}(n)V_{j+5}(p)}{\pi\omega(p)\sigma^j}
\biggl[1 +\frac{\widetilde\gamma_{j+5}e(n,p)n}{\sigma^{2}(n-2)} \biggr]
+\frac{\gamma_{6}\widetilde\gamma_{6}A_6(n)e(n,p)n}{\pi\omega(p)\sigma^{3}(n-2)}.
\end{align*}
Taking into account that for  $n\ge N_0$, $1\le j\le5$ and
$p\in[0.1689,0.5]$, we have
\begin{align*}
& A_{j+5}(n)<A_{10}(N_0)<e_3:=1.00001801,
\quad e(n,p)\le e(N_0,0.5)<1.02316,
\\
& 1+\frac{\widetilde\gamma_{j+5}\,e(n,p)\,n}{\sigma^{2}\,(n-2)}<1+\frac{(5/3)\cdot1.02316}{pq(N_0-2)}
\bigg|_{p=0.1689}<e_4:=1.0000243.
\end{align*}
Then, taking into account as well that $A_6(N_0)<1.0000101$, we get
\begin{align*}
r_2(p,n)<\widetilde r_2(p,n):=\frac{e_3\cdot
e_4}{\pi\omega(p)}\sum
_{j=2}^5\frac{\gamma_{j+5}V_{j+5}(p)}{\sigma^j}+
\frac{(1/9)(2/3)1.0000101\cdot1.02316}{\pi\omega(p)(pq)^{3/2}\sqrt{n}(n-2)}.
\end{align*}
We find with the help of a computer: $\widetilde r_2(p,n)\le
\widetilde r_2(0.1689,N_0)<8.852\cdot10^{-8}$.

Step 3. Let us write up
\begin{align*}
\frac{K_3(p,n)\sigma}{\omega(p)}=\frac{1}{12\pi\omega(p)\sigma}+r_3(p,n),
\end{align*}
where
\begin{align*}
r_3(p,n)&=\frac{\sigma}{\pi \omega(p)}\, \biggl\{ \biggl(
\frac{1}{36}+\frac{\mu}{8} \biggr)\, \frac{1}{\sigma^4}+ \biggl(
\frac{1}{36}\,e^{A_1/6}+ \frac{\mu}{8} \biggr)\,
\frac{1}{\sigma^6}+ \frac{5\mu}{24}\,e^{A_2/6}\, \frac{1}{\sigma^8}
\nonumber
\\
&\quad+\,\frac{1}{3}\,\exp \biggl\{ -\sigma\sqrt{A_1}+
\frac{A_1}{6} \biggr\}+(\pi-2)\mu\exp \biggl\{ -\sigma\sqrt{
A_2}+\frac{ A_2}{6} \biggr\}
\nonumber
\\
&\quad+\,\exp \biggl\{-\sigma\sqrt{A_3}+\frac{A_3}{6} \biggr
\} \frac{1}{4}\,\ln \biggl(\frac{\pi^4 \sigma^2}{4A_3} \biggr)
\nonumber
\\
&\quad+\,\exp \biggl\{-\frac{\sigma^{2/3}}{2\zeta(p)} \biggr\} \biggl[\frac{2\zeta(p)}{\sigma^{2/3}}
+e^{A_3/6}\,\frac{1+\chi(p,n)}{24\,\zeta(p)\,\sigma^{4/3}} \biggr] \biggr\}. 
\end{align*}
Using a computer, we get $r_3(p,n)\le
r_3(0.1689,N_0)<1.08\cdot10^{-9}$.

Thus, for $p\in[0.1689,0.5]$, $n\ge N_0$, we have
\begin{multline*}
r_1(p,n)+r_2(p,n)+r_3(p,n)<2.78
\cdot10^{-7}+8.852\cdot10^{-8}+1.08\cdot10^{-9}<3.676
\cdot10^{-7}.
\end{multline*}

Step 4. Now consider the function
\begin{align*}
B(p,n)= {\cal E}(p)+\frac{1}{12\pi\omega(p)\sigma} \biggl(\omega_4(p) \biggl(
\frac{n}{n-1} \biggr)^2 +12\gamma_{6}
\,A_{6}(n)\,V_{6}(p)+1 \biggr).
\end{align*}
We find with the help of a computer that for $p\in[0.1689,0.5]$,
$n\ge N_0$,
\begin{multline*}
\max_{p\in[0.1689,0.5]}B(p,n)=\max_{p\in[0.1689,0.5]}B(p,N_0)
\\
=B(0.418886928\ldots\;,N_0)=0.40995378459\ldots\;.
\end{multline*}
Consequently,
\begin{multline*}
E(p,n)=B(p,n)+\sum_{j=1}^3
r_j(p,n)
\\
 <0.4099537846+3.676\cdot10^{-7}
<0.409954153.
\qedhere
\end{multline*}
\end{proof}

Let us introduce the following notations:
\begin{align*}
{\cal E}_1(p)=\bigl(p^2+q^2\bigr){\cal E}(p)=
\frac{2-p}{3\sqrt{2\pi}},
\end{align*}
$D_2(p,n)$
is the coefficient at $\frac{1}{\sigma^2}$ in the expansion of
$R(p,n)$ in powers of $\frac{1}{\sigma}$,\break $\overline
D_2(p,n)=\sigma^2R(p,n)$, where the remainder $R(p,n)$ is defined by
equality~\eqref{R=K1+}. One can rewrite  bound \eqref{Delta-Th} in
the following form,
%
\begin{equation}
\label{Delta-Th-2} \Delta_n(p)\le\frac{{\cal E}_1(p)}{\sigma}+
\frac{\overline
D_2(p,n)}{\sigma^2}.
\end{equation} Define
$D_2^I(n)=\max\limits_{p\in I}D_2(p,n)$, $\overline
D_2^I(n)=\max\limits_{p\in I}\overline D_2(p,n)$, where $I$ is an
interval.

\begin{ccor}\label{corB}
The quantities $\max\limits_{n\ge N}D_2^I(n)$ and $\max\limits_{n\ge
N}\overline D_2^I(n)$ take the following values depending on
$N=200,\,N_0$ and intervals $I=[0.02,0.5]$, $[0.1689,0.5]$:

\begin{table}[h]
\caption{Some values of $\max\limits_{n\ge N}D_2^I(n)$ and
$\max\limits_{n\ge N}\overline D_2^I(n)$}
\label{tab2-D2}
\begin{tabular}{|c|c|c|c|}
 \hline
 &$\phantom{\int}$\hspace{-2mm}$I=[0.02,0.5]$&\multicolumn{2}{c|}{$I=[0.1689,0.5]$}\\[0.3mm]
 \hline&$N=200$&$N=200$&$N=N_0$\\[0.3mm]
\hline $\max\limits_{n\ge
N}D_2^I(n)=$&$0.083592\ldots$&$0.046656\ldots$&$0.0462198\ldots$\\[0.5mm]
\hline
$\max\limits_{n\ge N}\overline
D_2^I(n)=$&$0.1940\ldots$&$0.05986\ldots$&$0.05531\ldots$\\[0.5mm]
\hline
\end{tabular}
\end{table}
\end{ccor}

\begin{proof}
Since
\begin{align*}
\max_{n\ge N}\overline D_2^I(n)=\max
_{n\ge N}\max_{p\in
I}\sigma^2R(p,n)=
\max_{p\in I}\sigma^2R(p,N),
\end{align*} then by using a
computer, we get the tabulated values of $\max\limits_{n\ge
N}\overline D_2^I(n)$.

Proceed to the derivation of the values of $\max\limits_{n\ge N}
D_2^I(n)$. It follows from the definitions of $K_1(p,n)$,
$K_2(p,n)$, and $K_3(p,n)$ that the coefficient at
$\frac{1}{\sigma^2}$ in $R(p,n)$ is
\begin{align*}
D_2(p,n)=\frac{\omega_4(p)}{12\pi}\, \biggl(\frac{n}{n-1}
\biggr)^2+\frac{1}{\pi}\,\gamma_6
A_6(n)V_6(p)+\frac{1}{12\pi}
\end{align*} or, in more detail,
\begin{multline*}
D_2(p,n)=\frac{1}{36\pi} \biggl(3|q^3+p^3-3pq|
\biggl(\frac{n}{n-1} \biggr)^2 +4A_6(n)
(q-p)^2+3 \biggr)
\\
=:\frac{G_2(p,n)}{36\pi}.
\end{multline*}
First
 we consider ${G_2}(p):=\lim\limits_{n\to\infty}G_2(p,n)$. We have
\begin{align*}
{G_2}(p)=3|q^3+p^3-3pq|+4(q-p)^2+3
\equiv3|6p^2-6p+1|+4(1-2p)^2+3.
\end{align*}
 Taking into account that
\begin{align*}
|6p^2-6p+1|=\begin{cases}6p^2-6p+1&\text{if}\;p\le p_1:=\frac{3-\sqrt{3}}{6}=0.211324\ldots\;,\\
-6p^2+6p-1&\text{if}\;p>p_1,\end{cases}
\end{align*} we obtain
\begin{align*}
{G_2}(p)=\begin{cases}2(17p^2-17p+5)&\text{if}\;p\le p_1,
 \\
 -2(p^2-p-2)&\text{if}\;p> p_1.\end{cases} %
\end{align*}
Since $ G_2 (p) $ decreases for $ p <p_1 $, and increases for $ p>
p_1 $, then the maximum value of this function is achieved either at
the left bound or at the right bound of the interval. We have
\begin{align*}
G_2(0.02)=9.3336,\quad G_2(0.1689)=5.2273251\ldots\;,
\quad G_2(0.5)=4.5.
\end{align*} Thus,
\begin{align*}
&\frac{1}{36\pi}\,\max_{0.02\le
p\le0.5}G_2(p)=
\frac{G_2(0.02)}{36\pi}=0.0825271\ldots\;,
\\
&\frac{1}{36\pi}\,\max_{0.1689\le
p\le0.5}G_2(p)=
\frac{G_2(0.1689)}{36\pi}=0.04621970\ldots\;.
\end{align*}

Similarly, 
with more efforts only, we get
\begin{gather*}
\max_{0.02\le p\le0.5}G_2(p,200)= G_2(0.02,200)=9.4541
\ldots\;,
\\
\max_{0.1689\le p\le0.5}G_2(p,200)= G_2(0.1689,200)=5.2767
\ldots\;,
\\
G_2(0.5,200)=4.515\ldots\;,
\\
\max_{0.02\le p\le0.5}G_2(p,N_0)=
G_2(0.02,N_0)=9.33364\ldots\;,
\\
\max_{0.1689\le p\le0.5}G_2(p,N_0)=
G_2(0.1689,N_0)=5.227344\ldots\;,
\\
G_2(0.5,N_0)=4.00006\ldots\;.
\end{gather*}

  Consequently,
\begin{gather*}
\frac{\max\limits_{0.02\le p\le0.5}G_2(p,200)}{36\pi}= 0.083592\ldots\;,\quad \frac{\max\limits_{0.1689\le
p\le0.5}G_2(p,200)}{36\pi}=0.046656\ldots
\;,
\\
\frac{\max\limits_{0.1689\le
p\le0.5}G_2(p,N_0)}{36\pi}=0.0462198\ldots\;.
\qedhere
\end{gather*}
\end{proof}

\begin{remark} 1.  One can observe from the previous proof that
$G_2(p,N_0)\approx G_2(p)$, therefore,
$D_2(p,N_0)\approx\frac{G_2(p)}{36\pi}$.

2. With increasing $N$, the sequence $a^I(N):=\max\limits_{n\ge N}
 D_2^I(n)$
approaches to $a^I:=\frac{1}{36\pi}\,\max\limits_{p\in I}G_2(p)$.
For instance, by Table~\ref{tab2-D2}, we have for the interval
$I=[0.1689,0.5]$ that $a^I(200)=0.046656\ldots\;$,
$a^I(N_0)=0.0462198\ldots\;$ while $a^I=0.0462197\ldots \;$.  The
sequence $\overline a^I(N):=\max\limits_{n\ge N} \overline D_2^I(n)$
tends to $0.0462197\ldots \;$ as well, but slowly,  since the main
term of the difference $\overline D_2(p,n)-\frac{G_2(p)}{36\pi}$ has
the order $\frac{1}{\sqrt{n}}$.
\end{remark}

The following bound for $\Delta_n(p)$, simpler than Theorem~\ref{thmA},
follows from \eqref{Delta-Th-2} and Table~\ref{tab2-D2}.

\begin{ccor}\label{corC}
 For all $p\in I=[0.1689,0.5]$
and $n\ge N_0$,
%
\begin{equation}
\label{532}\Delta_n(p)\le\frac{{\cal
E}_1(p)}{\sigma}+\frac{0.05532}{\sigma^2}.
\end{equation}
\end{ccor}

\begin{remark}
Corollary~C allows to obtain the same estimate for $ C_{02} $
as~\eqref{953}, but for larger $ n $. Really, it is easy to verify
with the help of a computer that
%
\begin{equation}
\label{954-4} \sup_{p\in[0.1689,0.5]} \biggl({\cal E}(p)+
\frac{0.05532}{\sqrt{npq}(p^2+q^2)} \bigg|_{n=971000} \biggr) <0.409954,
\end{equation} but
%
\begin{equation}
\label{954-3} \sup_{p\in[0.1689,0.5]} \biggl({\cal E}(p)+
\frac{0.05532}{\sqrt{npq}(p^2+q^2)} \bigg|_{n=970000} \biggr)>0.409954.
\end{equation}
\end{remark}

\subsection{On the connection between Uspensky's result and its refinements with the problem of estimating $C_{02}$}

First we recall  Uspensky's estimate, published by him in 1937 in
\cite{Uspen}. To this end we introduce the following notations:
$S_n$ is the number of occurrences of an event in a series of $n$
Bernoulli trials with a  probability of success  $p$, $\mu=np$,
\begin{align*}
G(x)=\varPhi(x)+\frac{q-p}{6\sqrt{2\pi}\,\sigma}\bigl(1-x^2\bigr)e^{-x^2/2}.
\end{align*}
 For
every $x\in{\mathbb R}$, define
%
\begin{equation}
\label{x^+}x_n^\pm=\frac{x-\mu\pm\frac{1}{2}}{\sigma},
\end{equation}
where $\sigma=\sqrt{npq}$, as before.

 Uspensky's result can be formulated in the following form.

\begin{tthm}[{\cite[p. 129]{Uspen}}]
\label{thmB}
Let $\sigma^2\ge25$. Then
for arbitrary integers $a<b$,
%
\begin{equation}
\label{resU-1} \big|{\bf P}(a\le S_n\le b)- \bigl(G
\bigl(b_n^- \bigr)-G \bigl(a_n^+ \bigr) \bigr) \big|\le
\frac{0.13+0.18|p-q|}{\sigma^2}+e^{-3\sigma/2}.
\end{equation}
\end{tthm}

A lot of works are devoted to generalizations and refinements  of
\eqref{resU-1}, for example,
\cite{Deh,Makabe-1961,Makabe-1955,Mikhailov-93,Neamm,Senatov-2014-Eng,Volkova-95}.

In 2005 K.~Neammanee \cite{Neamm} refined and generalized~\eqref{resU-1} to the case of non-identically
distributed Bernoulli random variables. Let us formulate his result
as applied to the case of Bernoulli trials: {\it if
$\sigma^2\ge100$, then
%
\begin{equation}
\label{Neam-1} \big|{\bf P}(a\le S_n\le b)- \bigl(G
\bigl(b_n^- \bigr)-G \bigl(a_n^+ \bigr) \bigr) \big|<
\frac{0.1618}{\sigma^2},
\end{equation}
where $a_n^+$, $b_n^-$ are defined by the formula  \eqref{x^+}}.\querymark{Q2}

It follows from \eqref{Neam-1} that under condition
$\sigma^2\ge100$,
%
\begin{equation}
\label{Neam-20} \big|{\bf P}(S_n\le b)- G \bigl(b_n^- \bigr)\big|
\le\frac{0.1618}{\sigma^2}.
\end{equation}

We may consider $p\in(0,0.5]$. Denote for brevity, $d=0.1618$. It
follows from \eqref{Neam-20} and the definition of $G(\cdot)$ that
\begin{equation*}
\big|{\bf P}(S_n\le b)-\varPhi \bigl(b_n^- \bigr) \big|<
\frac{|(1-(b_n^-)^2)(q-p)|e^{-(b_n^-)^2/2}}{6\sqrt{2\pi
}\sigma}+\frac{d}{\sigma^2}.
\end{equation*}
 Taking into account that
$\max\limits_{t}|t^2-1|e^{-t^2/2}=1$, we get
%
\begin{equation}
\label{Neam-3} \big|{\bf P}(S_n\le b)-\varPhi \bigl(b_n^-
\bigr) \big|\le \frac{|q-p|}{6\sqrt{2\pi}\sigma}+\frac{d}{\sigma^2} .
\end{equation}

Denote $x_n=\frac{x-\mu}{\sigma}$. It is easily seen that
%
\begin{equation}
\label{Phi-Phi-Neam} \big|\varPhi(b_n)-\varPhi \bigl(b_n^- \bigr)
\big|< \frac{b_n-b_n^-}{\sqrt{2\pi}}=\frac{1}{2\sqrt{2\pi
}\sigma}.
\end{equation} It follows from \eqref{Neam-3},
\eqref{Phi-Phi-Neam} that
\begin{align*}
\big|{\bf P}(S_n\le b)-\varPhi(b_n)\big|< \biggl(
\frac{|q-p|}{6}+\frac{1}{2} \biggr) \frac{1}{\sigma\sqrt{2\pi}}+
\frac{d}{\sigma^2}=\frac{{\cal
E}_1(p)}{\sigma}+\frac{d}{\sigma^2}, %
\end{align*}
provided that $0<p\le0.5$. Thus,
%
\begin{equation}
\label{P-Phi} \Delta_n(p)\le\frac{{\cal
E}_1(p)}{\sigma}+
\frac{0.1618}{\sigma^2}.
\end{equation} Note that
our bound \eqref{532} is  more accurate than \eqref{P-Phi}. To get
the bound $0.409954$ for $C_{02}$ from~\eqref{P-Phi}, we should take
$n$ almost five times larger than in~\eqref{954-4}. Really, with the
help of a computer we have
\begin{align*}
\sup_{p\in[0.1689,0.5]} \biggl( {\cal E}(p)+\frac{0.1618}{\sqrt{npq}\,(p^2+q^2)}
\bigg|_{n=4.6\cdot10^6} \biggr)<0.410031,
\end{align*}
and
\begin{align*}
\sup_{p\in[0.1689,0.5]} \biggl( {\cal E}(p)+\frac{0.1618}{\sqrt{npq}\,(p^2+q^2)}
\bigg|_{n=4.2\cdot10^6} \biggr)>0.410044
\end{align*}
(cf. \eqref{954-4}, \eqref{954-3}).


\begin{remark}
In 2014 V. Senatov obtained non-uniform estimates of the
approximation accuracy in the central limit theorem, and, in particular,
generalized Uspensky's result \eqref{resU-1} to lattice
distributions \cite{Senatov-2014-Eng}.
\end{remark}

\subsection{Proof of Theorem \ref{th-2}}

Before proving Theorem \ref{th-2}, we first prove
Lemma~\ref{lem-1-ZNC}.

\begin{proof}[Proof of Lemma \ref{lem-1-ZNC}] By  \cite[Theorem
1]{KorShv-Obozr-2010-2},
%
\begin{equation}
\label{var-1}\Delta_n(p)\le\frac{0.33477}{\sqrt{n}}\, \bigl(
\varrho(p)+0.429 \bigr).
\end{equation}
Therefore,
$T_n(p)\equiv\frac{\sqrt{n}\,\Delta_n(p)}{\varrho(p)}\le0.33477 (1+\frac{0.429}{\varrho(p)} )$.
Since $\varrho(p)$ decreases on $(0, 0.5]$, then
$\max\limits_{p\in(0,0.1689]}\frac{1}{\varrho(p)}=\frac{1}{\varrho(0.1689)}=0.52090548\ldots\;\,$.
Consequently,
\begin{equation*}
\max\limits
_{p\in(0,0.1689]}T_n(p)\le 0.33477(1+0.429
\cdot0.52090549)<0.409581.
\qedhere
\end{equation*}
\end{proof}

\begin{remark} If instead of \cite[Theorem
1]{KorShv-Obozr-2010-2} we will use other modifications of the
Berry--Esseen inequality  by I.~Shevtsova
\cite{arxive-Shevtsova-2013}, the interval (0,0.1689] for which
Lem\-ma~\ref{lem-1-ZNC} is true can be extended, i.e. one can find $
b> 0.1689 $ such that the inequality $\max\limits_{p\in(0,
b]}T_n(p)<C_E $ will be fulfilled. This will narrow the interval $ I
$ (see~\eqref{<0.4-S}), which in turn will reduce the computation
time on the supercomputer.

Let us indicate such $b$. The estimates found in
\cite{arxive-Shevtsova-2013} as applied to the particular case of
Bernoulli trials can be written in the following form,
%
\begin{align}
\label{var-2}&\Delta_n(p)\le\frac{0.33554}{\sqrt{n}}\, \bigl(
\varrho(p)+0.415 \bigr),
\\
\label{var-3}&\Delta_n(p)\le\frac{0.3328}{\sqrt{n}}\, \bigl(
\varrho(p)+0.429 \bigr).
\end{align}
It is easy to verify that inequality \eqref{var-2} implies $b
=0.174$, and \eqref{var-3} implies that $b =0.177$.
\end{remark}

\begin{proof}[Proof of Theorem \ref{th-2}] It follows from
Corollary~\ref{th-1.2} and
 \eqref{<0.4-S} that for all  $p\in I$ the following bound holds,
%
\begin{equation}
\label{Knp<0.41}T_n(p)<0.40973213+4.6\cdot10^{-9}<
0.4097321346,\quad 1\le n\le N_0.
\end{equation}
Then by Lemma~\ref{lem-1-ZNC}, this inequality is fulfilled  for all
$p\in(0,0.5]$ as well. It is not hard to see that the bound
\eqref{Knp<0.41} is also true for all $p\in(0.5,1)$. Hence,
bound~\eqref{954} implies Theorem~\ref{th-2}.
\end{proof}


\begin{acknowledgement}[title={Acknowledgments}]
We thank the following colleagues from Lomonosov Moscow State
University  for providing the opportunity to use supercomputer Blue
Gene/P:    V.~Yu.~Korolev, Head of the Department of Mathematical
Statistics of the Faculty of Computational Mathematics and
Cybernetics, Professor, I.~G.~Shevtsova, Assistant Professor of the
same Department,  A.~V.~Gulyaev, Deputy Dean of the same Faculty,
and S.~V.~Korobkov, the Data Center administrator.

We also thank our colleagues from Computing Center FEB RAS for the
opportunity to use the Center for the Collective Use ``Data Center
FEB RAS''.

We also would like to thank reviewers for useful comments.
\end{acknowledgement}


\end{document}